\documentclass{amsart}
\usepackage{xcolor}
\usepackage{amsmath,amssymb,amsthm,amsfonts,amscd,mathrsfs,mathtools,enumerate}
\usepackage{enumitem}
\usepackage{tikz,graphicx}
\usepackage[all]{xy}
\usepackage[hyphens]{url}
\usepackage{hyperref}
\usepackage{caption}
\usepackage{tikz-cd}
\usepackage{booktabs}
\usepackage{graphicx}
\usepackage{comment}
\usetikzlibrary{decorations.markings}
\newcommand{\an}[1]{\langle{#1}\rangle}

\theoremstyle{plain}
    \newtheorem{question}{Question}[section]    
    \newtheorem{lem}[question]   {Lemma}
    \newtheorem{cor}[question]   {Corollary}
    \newtheorem{prop}[question]  {Proposition}
  
    \newtheorem{thmA}{Theorem}
    
\theoremstyle{definition}
  \newtheorem{thm}[question]{Theorem}

    \newtheorem{rem}[question]{Remark}

\newcommand{\imm}{\looparrowright}
\newcommand{\be}{\begin{enumerate}}
\newcommand{\ee}{\end{enumerate}}
\newcommand{\R}{\mathbb{R}}
\newcommand{\Z}{\mathbb{Z}}

\newcommand{\Th}{\textup{Th}}

\newcommand{\la}{\langle}
\newcommand{\ra}{\rangle}
\newcommand{\symp}{\operatorname{Sp}_{2}}
\newcommand{\eps}{\varepsilon}
\newcommand{\xra}{\xrightarrow}

\usetikzlibrary{calc}

\usetikzlibrary{decorations.pathmorphing}

\tikzset{curve/.style={settings={#1},to path={(\tikztostart)
    .. controls ($(\tikztostart)!\pv{pos}!(\tikztotarget)!\pv{height}!270:(\tikztotarget)$)
    and ($(\tikztostart)!1-\pv{pos}!(\tikztotarget)!\pv{height}!270:(\tikztotarget)$)
    .. (\tikztotarget)\tikztonodes}},
    settings/.code={\tikzset{quiver/.cd,#1}
        \def\pv##1{\pgfkeysvalueof{/tikz/quiver/##1}}},
    quiver/.cd,pos/.initial=0.35,height/.initial=0}

\tikzset{tail reversed/.code={\pgfsetarrowsstart{tikzcd to}}}
\tikzset{2tail/.code={\pgfsetarrowsstart{Implies[reversed]}}}
\tikzset{2tail reversed/.code={\pgfsetarrowsstart{Implies}}}

\tikzset{no body/.style={/tikz/dash pattern=on 0 off 1mm}}

\usepackage{color}
\newcommand{\remdc}[1]{\begingroup\color{blue}#1\endgroup}

\begin{document}

\title[Immersed not embedded]{Immersed but not embedded homology classes}

\author{Diarmuid Crowley}
\author{Mark Grant}

\address{School of Mathematics and Statistics,
The University of Melbourne,
Parkville, VIC, 3010,
Australia}
\address{Institute of Mathematics,
Fraser Noble Building,
University of Aberdeen,
Aberdeen AB24 3UE,
UK}

\email{dcrowley@unimelb.edu.au}
\email{mark.grant@abdn.ac.uk}

\date{\today}

\keywords{immersion, embedding, (co)homology, Thom space}

\begin{abstract} 
We provide the first 
documented examples of immersions of closed oriented manifolds 
which are not homologous to embeddings, thus answering a question posed by Zhenhua Liu. 
In these examples
we show that for any representing self-transverse immersion the
double points must represent a non-trivial homology class in the source manifold.
We also provide examples of Steenrod representable integral homology classes which are not represented by immersions.

\end{abstract}


\maketitle

\enlargethispage{1\baselineskip}

\section{Introduction} \label{sec:intro}
In this paper we take up the classical question of representing positive codimension integral homology 
classes in an oriented manifold by maps of varying regularity from another oriented manifold.
Throughout, all manifolds are assumed smooth and oriented, (co)homology is 
integral when coefficients are omitted and ``$M$" and ``$N$" will denote closed, connected manifolds.
Let $N$ be an $n$-manifold and $z \in H_{n-k}(N)$ a homology class.
We say that $z$ is {\em Steenrod representable} if there is a closed oriented $(n{-}k)$-manifold $M$
and a map $f \colon M \to N$ such that $f_*[M] = z$, where $M \in H_{n-k}(M)$ is the fundamental class of $M$.
We say that $z$ is {\em immersed} if $f$ can be chosen to be an immersion and {\em embedded} if
$f$ can be chosen to be an embedding.

Given $z \in H_{n-k}(N)$ as above, in various applications one may wish to know whether $z$ is Steenrod representable, immersed or embedded.  There is a good deal of work on these questions, which we review in Section~\ref{ss:bg} below.
However, to our surprise it does not seem that examples distinguishing these three cases have been documented in the literature,
and the purpose of this paper is to give such examples.
Our work was prompted by the following question of Zhenhua Liu on MathOverflow \cite{Liu}:

\begin{question}[Liu]
Does there exist an integral homology class of positive codimension in an orientable manifold, which is immersed but not embedded?
\end{question}
%
%
\noindent
Our first two results, Theorems A and B, show that the answer to Liu's question is ``\emph{yes}".
Our fourth result, Theorem D, shows that there are Steenrod representable classes which
are not immersed.

Let $\symp$ denote the compact Lie group of invertible $2 \times 2$ quaternionic matrices preserving the standard Hermitian form on $\mathbb{H}^2$. Its underlying smooth manifold is the total space
of a principal $S^3$-bundle over $S^7$.  Hence $H_*(\symp) \cong H_*(S^3 \times S^7)$.

\begin{thmA}\label{thm1}
Both generators of $H_7(\symp)$ are immersed but not embedded.
\end{thmA}


\begin{cor}
For all $n\ge10$ there exists a closed oriented $n$-manifold with an $(n{-}3)$-dimensional homology class which is immersed but not embedded.
\end{cor}

\begin{proof}
For $n>10$ let $N^n=\symp\times S^{n-10}$ and consider $\zeta:=z\times[S^{n-10}]\in H_{n-3}(N)$, where $z\in H_7(\symp)$ is a generator. Let $f:M^7\imm \symp$ be an immersion representing $z$. Then $f\times \operatorname{id}:M\times S^{n-10}\imm N$ is an immersion representing $\zeta$, so $\zeta$ is immersed. 

On the other hand, if $\zeta$ were represented by an embedding $Z^{n-3}\hookrightarrow N$, then by taking transverse intersection with the inclusion $\symp\hookrightarrow N$ we would obtain an embedding $Z\cap\symp\hookrightarrow \symp$ representing $z$, a contradiction.
\end{proof}

\begin{rem}
Bohr, Hanke and Kotschick have already shown that $z$ is not embedded in \cite[Theorem 1]{BHK}, so the novel part of Theorem \ref{thm1} is that $z$ is immersed, but see Theorem~\ref{t:doublesymp} for 
an alternative proof that $z$ is not embedded.
As noted in \cite[Remark 2]{BHK}, since $H_*(\symp)$ is torsion free, 
every class in $H_*(\symp)$ is Steenrod representable. 
However, as we prove in Theorem \ref{thm4} below, 
a homology class being Steenrod representable does not entail its being immersed.
(The terminology ``immersed submanifold" in \cite[Remark 2]{BHK} is perhaps unfortunate.)
The examples in Theorem \ref{thm4} are $2$-torsion homology classes.
We do not know whether every Steenrod representable homology class of infinite order is immersed, nor whether
every homology class in a manifold with torsion-free homology is immersed.
\end{rem}

Our next result is about the total spaces $N$ of certain linear $S^{11}$-bundles over $S^{13}$.
The total spaces of these bundles were classified by Ishimoto \cite{Ishimoto}, who proved
that up to connected sum with homotopy $24$-spheres, 
there is a unique diffeomorphism type $N$ such that $Sq^2 \colon H^{11}(N; \Z_2) \to H^{13}(N; \Z_2)$ is non-zero.

\begin{thmA}\label{thm2}
Let $N$ be the total space of a linear $S^{11}$-bundle over $S^{13}$ such that $Sq^2 \colon H^{11}(N; \Z_2) \to H^{13}(N; \Z_2)$
is non-zero.  Then both generators of $H_{13}(N)$ are immersed but not embedded.
\end{thmA}

Our proof of Theorem \ref{thm2} uses Theorem~\ref{thm3}, which
identifies an obstruction to embeddability which does not obstruct immersability.
This obstruction was initially discovered through a careful analysis and comparison 
of the Postnikov towers of the Thom space $MSO_{11}$ and $QMSO_{11}$, where $QX:=\lim_{\ell \to \infty} \Omega^\ell\Sigma^\ell X$ 
denotes the free infinite loopspace on a space $X$; see Section~\ref{ss:bg}. (Below $\rho_p:H^*(-)\to H^*(-;\Z_p)$ denotes mod $p$ reduction, and $P^i_p:H^*(-) \to H^{*+2i(p-1)}(-;\Z_p)$ denotes mod $p$ reduction followed by the $i$-th Steenrod power.)
 
\begin{thmA}\label{thm3}
Let $z\in H_{n-11}(N)$ be an embedded homology class in the closed smooth oriented $n$-manifold $N$, and let $x=PD(z)\in H^{11}(N)$ be its Poincar\'e dual cohomology class. Then there exist cohomology classes
\[
 \alpha\in H^{20}(N;\Z_2), \quad \beta,\beta' \in H^{19}(N),\quad \gamma \in H^{16}(N;\Z_2),\quad \delta\in H^{15}(N)
 \]
 such that 
 \[
 \rho_5(\beta)=P^1_5 x,\quad \rho_3(\beta')=P^2_3 x, \quad \rho_3(\delta)=P^1_3 x
 \]
 and
 \[
 (Sq^4+ Sq^3Sq^1)\alpha + Sq^5\beta + (Sq^8 + Sq^7Sq^1 + Sq^6Sq^2)\gamma + Sq^9 \delta = xSq^2 x + Sq^{11}Sq^2 x.
 \]
\end{thmA}

We conclude this summary of our main results by stating Theorem \ref{thm4}, which shows that a homology class being Steenrod representable does not imply its being immersed.

\begin{thmA}\label{thm4}
For all $n \geq 27$, there exist closed oriented $n$-manifolds $N$ and $2$-torsion classes $z \in H_{n-4}(N)$,
which are Steenrod representable but not immersed.
\end{thmA}

\subsection{Background} \label{ss:bg}
The classical Steenrod representability problem asks the following: given 
a finite complex $X$ and a homology class $z \in H_m(X)$,
does there exist a closed oriented $m$-manifold $M$ with fundamental class $[M]\in H_m(M)$ and a continuous map $f\colon M\to X$ such that $z=f_*[M]$? In his remarkable 1954 paper \cite{Thom} Ren\'e Thom showed that not every integral homology class is so representable, but that every class is representable after multiplication by a positive 
integer. (Finding universal bounds on this integer for a fixed $m$ is a subject of current research, see \cite{Roz}.) Thom also showed that for the analogous problem in mod $2$ homology and unoriented manifolds, every homology class is representable. In both cases, Thom's method was to relate Steenrod's problem to the problem of representability by embeddings in manifolds, which he then interpreted as a lifting problem. 

Recall that a homology class $z\in H_{n-k}(N)$ in a smooth $n$-manifold $N$ is \emph{embedded} if there exists a closed oriented $(n-k)$-manifold $M$ with fundamental class $[M]\in H_{n-k}(M)$ and a smooth embedding $f\colon M\hookrightarrow N$ such that $z=f_*[M]$. Since $N$ is oriented, $z$ admits a Poincar\'e dual cohomology class $PD(z)\in H^k(N)\cong\tilde{H}^k(N_+)$. Thom's key insight was that there exists a universal embedded cohomology class in codimension $k$. Namely, he constructed a space $MSO_k$ with a cohomology generator $t_k\in \tilde{H}^k(MSO_k)$ such that $z$ is embedded if and only if $PD(z)$ is the image of $t_k$ under the map in cohomology induced by some pointed map $N_+\to MSO_k$. Put another way, $z$ is embedded if and only if the map $N_+\to K(\Z,k)$ representing $PD(z)$ admits a lift up to homotopy through the map $MSO_k\to K(\Z,k)$ representing $t_k$. (A similar statement holds in the unoriented case, using the Thom class $MO_k\to K(\Z_2,k)$.) 

This allowed Thom to prove both positive and negative embeddability results, using methods which are by now classical; namely, Postnikov towers and cohomology operations. For example, since all torsion in the cohomology of $MSO_k$ is $2$-torsion, it follows that for any odd prime $p$ and $i\ge1$, the operation $\beta P^i_p$ annihilates the Thom class $t_k$. (Here $\beta\colon H^*(-;\Z_p)\to H^{*+1}(-)$ is the Bockstein operator.) Hence the same must be true for any embedded cohomology class. In this way he was able to find a $7$-dimensional homology class in a $14$-manifold (a product of lens spaces) which is not embedded, and in fact not even Steenrod representable \cite[pp.62--63]{Thom}. On the positive side, by analysing the Postnikov tower of $MSO_k$ for small values of $k$, he was able to show that any homology class in a closed oriented manifold of dimension at most $9$ is embedded. The smallest dimensional case left open by Thom was determining whether there exists a $10$-manifold with a $7$-dimensional non-embedded homology class. This was eventually settled by Bohr, Hanke and Kotschick \cite{BHK} who showed that the generator of $H_7(\symp)$ is not embedded. We will return to this example in much greater detail in Section \ref{S:proofA}. 

There is an analogous obstruction-theoretic interpretation of the problem of representability by immersions. A homology class $z\in H_{n-k}(N)$ is \emph{immersed} if there exists a closed oriented $(n{-}k)$-manifold $M$ and an immersion $f\colon M\imm N$ such that $f_*[M]=z$. By the Compression Theorem of Rourke and Sanderson \cite{RS1,RS3}, for $k\ge2$ and $\ell\gg0$ immersions $f\colon M\imm N$ up to regular homotopy correspond to embeddings $g\colon M\hookrightarrow N\times \R^\ell$ with normal $\ell$-fields up to isotopy. Combining this with the insights  of Thom, one can show that $z$ is immersed if and only if $PD(z)$ is the image of $t_k$ under some \emph{stable} map $N_+\rightsquigarrow MSO_k$. More precisely, $z$ is immersed if and only if there is a map $\varphi\colon\Sigma^\ell N_+\to \Sigma^\ell MSO_k$ for some $\ell\ge0$ such that $PD(z)$ is the image of $t_k$ under the composition $\tilde{H}^k(MSO_k)\cong \tilde{H}^{k+\ell}(\Sigma^\ell MSO_k)\stackrel{\varphi^*}{\to}\tilde{H}^{k+\ell}(\Sigma^\ell N_+)\cong H^k(N)$.  Using standard adjunctions, this can be rephrased as follows: for a pointed space $X$ let $QX:=\lim_{\ell \to \infty} \Omega^\ell\Sigma^\ell X$ be the free infinite loopspace on $X$. As the Thom class $t_k\colon MSO_k\to K(\Z,k)$ is a map with target an infinite loopspace, it extends to an infinite loop map $\tilde{t}_k\colon QMSO_k\to K(\Z,k)$. This map deserves to be called the universal Thom class for immersions, since $z$ is immersed if and only if the map $N_+\to K(\Z,k)$ representing $PD(z)$ admits a lift up to homotopy through $\tilde{t}_k\colon QMSO_k\to K(\Z,k)$. 

The ideas in the previous paragraph can be traced back to the thesis of Robert Wells \cite{Wells} on cobordism groups of immersions. They were applied by the second author and Sz\H{u}cs \cite{GSz} in the unoriented case to give the first examples of mod $2$ homology classes not representable by immersions. Generally speaking though, less is known about representability by immersions than representability by embeddings. In particular, we are not aware of any systematic anlaysis of the Postnikov tower of $QMSO_k$ in order to prove immersability results.

\subsection*{Outline}
In Section \ref{S:proofA} we prove Theorem \ref{thm1}. In Section \ref{S:proofsBC} we prove Theorems \ref{thm2} and \ref{thm3}, and outline an alternative approach using Postnikov towers. In Section \ref{S:SRnotim} we prove Theorem~\ref{thm4}.
In Section \ref{S:Whitney} we consider the double-point manifolds of self-transverse immersions, giving a
new proof of the non-embedding statement of Theorem~\ref{thm1}, and
we present examples of formal immersions representing the homology classes appearing in Theorems \ref{thm1} and \ref{thm2}.

\subsection*{Acknowledgements}
We wish to acknowledge Zhenhua Liu's MathOverflow post and the discussion there, which motivated and informed this work.
In particular the comment by Lennart Meier to this post drew our attention to the generators of $H_7(\symp)$ and to the work of Bohr, Hanke and Kotschick \cite{BHK}.

\section{Proof of Theorem A}\label{S:proofA}
Recall that the symplectic group $\symp$ is the compact Lie group of $2\times 2$ quaternionic matrices preserving the Hermitian form $\langle x,y\rangle =\overline{x_1}y_1+\overline{x_2}y_2$ on $\mathbb{H}^2$. The columns of such a matrix are unit norm vectors in $\mathbb{H}^2$, and thus projection onto the first column yields a principal bundle
\[
\xymatrix{
S^3=Sp(1) \ar[r] & \symp \ar[r] & S^7.
}
\]
Although the spectral sequence of this bundle collapses for reasons of dimension, the bundle is non-trivial. In fact, it is well known that $\symp$ has a minimal cell structure of the form
\[
\symp\simeq (S^3\cup_{\omega'} e^7)\cup_f e^{10},
\]
where the attaching map of the $7$-cell is the Blakers--Massey generator $\omega'\in \pi_6(S^3)\cong\Z_{12}$ \cite[pp. 201--202]{JW}. 
Since $\omega'$ is detected by the Steenrod power $P^1_3$, the operation $P^1_3\colon H^3(\symp; \Z_3) \to H^7(\symp; \Z_3)$ is non-trivial, showing that $\symp$ is not a product of spheres \cite[Cor.\ 13.5]{BorelSerre}.  

The Steenrod power $P^1_3$ was also used by Bohr--Hanke--Kotschick \cite{BHK} to show that a generator $z\in H_7(\symp)$ is not embedded. We now briefly recall their argument. Supposing $z$ is embedded, the Poincar\'e dual class $x:=PD(z)\in H^3(\symp)$ equals $f^*(t)$ for some map $f\colon (\symp)_+\to MSO(3)$, where $t\in \widetilde{H}^3(MSO(3))$ is the Thom class. Since $H^*(BSO(3);\Z_3)\cong\Z_3[p_1]$ is polynomial on the mod $3$ reduction of the Pontryagin class, $\tilde{H}^*(MSO(3);\Z_3)$ is concentrated in degrees $4s{+}3$. Therefore $tP^1_3(t)\in H^{10}(MSO(3);\Z_3)=0$. By naturality, it follows that $0=xP^1_3(x)\in H^{10}(\symp;\Z_3)$. This gives a contradiction, since by Poincar\'e duality with mod $3$ coefficients the generators $x\in H^3(\symp;\Z_3)$ and $P^1_3(x)\in H^7(\symp;\Z_3)$ must pair non-trivially.

Thus to complete the proof of Theorem \ref{thm1}, it remains to prove the following.

\begin{prop} \label{Sp2}
Each generator $z\in H_7(\symp)$ is immersed.
\end{prop}

\begin{proof}
Since $\symp$ is a Lie group it is parallelizable, hence stably parallelizable. It follows that the top cell splits off stably; in fact, by a result of Mimura \cite{Mimura}, two suspensions is enough, and we have
\[
\Sigma^2 \symp \simeq (S^5\cup_{\Sigma^2\omega'} e^9)\vee S^{12} = \Sigma^2 K\vee S^{12},
\]
where $K=S^3\cup_{\omega'}e^7$ denotes the $7$-skeleton of $\symp$. 

The $J$-homomorphism $\pi_3(SO(3))\to \pi_6(S^3)$ being onto, there exists a rank $3$ vector bundle $\xi\to S^4$ such that 
$K\simeq \Th(\xi)$, the Thom space of $\xi$.  (Here we use the fact that the Thom space of a rank $q$ vector bundle over $S^{p+1}$ with clutching function $\alpha\in \pi_p(SO(q))$ is homotopy equivalent to the mapping cone $S^q \cup_{J(\alpha)} e^{p+q+1}$ of the image of $\alpha$ under $J$. See \cite[Lemma 1]{Milnor}; a proof in the complex case appears in the undergraduate thesis of Eva Belmont \cite[Lemma 2.4.3]{Belmont}.)

A generator $z\in H_7(\symp)$ is immersed if and only if there is a stable map $f\colon(\symp)_+\rightsquigarrow MSO(3)$ such that $f^*(t)=x$, where $x=PD(z)\in H^3(\symp)$ and $t\in H^3(MSO(3))$ is the Thom class. Let $x|_K\in H^3(K)\cong\Z$ denote the restriction of $x$ to the $7$-skeleton. As $x|_K$ is a generator, it must be a Thom class $t(\xi)$. Hence there is a map $f_K\colon K\to MSO(3)$ such that $f_K^*(t)=x|_K$.  Now by the above splitting we have a stable map
\[
\xymatrix{
\Sigma^2\symp \simeq \Sigma^2 K\vee S^{12} \ar[rr]^-{\Sigma^2f_K \vee 0} && \Sigma^2MSO(3),
}
\]
under which the suspension of the Thom class pulls back to the suspension of $x$.
 \end{proof} 

\section{Proof of Theorems B and C}\label{S:proofsBC}

We begin by proving Theorem \ref{thm3}, which is restated here for convenience.

\begin{thm}
Let $z\in H_{n-11}(N)$ be an embedded homology class in the closed smooth oriented $n$-manifold $N$, and let $x=PD(z)\in H^{11}(N)$ be its Poincar\'e dual cohomology class. Then there exist cohomology classes
\[
 \alpha\in H^{20}(N;\Z_2), \quad \beta,\beta' \in H^{19}(N),\quad \gamma \in H^{16}(N;\Z_2),\quad \delta\in H^{15}(N)
 \]
 such that 
 \[
 \rho_5(\beta)=P^1_5 x,\quad \rho_3(\beta')=P^2_3 x, \quad \rho_3(\delta)=P^1_3 x
 \]
 and
 \[
 (Sq^4+ Sq^3Sq^1)\alpha + Sq^5\beta + (Sq^8 + Sq^7Sq^1 + Sq^6Sq^2)\gamma + Sq^9 \delta = xSq^2 x + Sq^{11}Sq^2 x.
 \]
 \end{thm}
 
 \begin{proof}
As $z$ is embedded there exists a map $f\colon N_+\to MSO_{11}$ such that $f^*(t)=x$, where $t\in \tilde{H}^{11}(MSO_{11})$ is the Thom class. In the cohomology of $MSO_{11}$ we have the following classes, where we have identified $\tilde{H}^*(MSO_{11};\Z_2)$ with the ideal in $H^*(BSO_{11};\Z_2)$ generated by $w_{11}$:
 \[
 \tilde{\alpha}=w_{11}w_6w_3,\quad \tilde{\beta}=t(p_1^2-2p_2),\quad \tilde{\beta'}=t p_2,\quad \tilde{\gamma} = w_{11}w_3w_2,\quad \tilde{\delta} = t p_1.
 \] 
According to \cite[Theorem 19.7]{MS}, attributed to Wu Wen-Tsun \cite{Wu}, one has
\[
\rho_5(\tilde{\beta})=P^1_5 t,\quad \rho_3(\tilde{\beta'})=P^2_3 t,\quad \rho_3(\tilde{\delta})=P^1_3 t.
\]
It therefore suffices to check that the equation
\[
 (Sq^4+ Sq^3Sq^1)\tilde\alpha + Sq^5\tilde\beta + (Sq^8 + Sq^7Sq^1 + Sq^6Sq^2)\tilde\gamma + Sq^9 \tilde\delta = tSq^2 t + Sq^{11}Sq^2 t
 \]
 holds in $H^{24}(MSO_{11};\Z_2)$, for then by setting $\alpha=f^*(\tilde\alpha)$, $\beta=f^*(\tilde\beta)$ etc.\ we obtain classes in the cohomology of $N$ which by naturality satisfy the stated properties.

First observe that $\rho_2(t)=w_{11}$, and by the formulae of Wu and Cartan we have
\[
w_{11}Sq^2 w_{11} = w_{11}(w_{11}w_2),\quad Sq^{11}Sq^2 w_{11}=w_{11}(w_{11}w_2 + w_{10}w_3 + w_9w_2^2).
\]
Therefore the right-hand side equals $w_{11}(w_{10}w_3 + w_9w_2^2)$.
For the left-hand side we sum up the elements
\begin{align*}
Sq^4 \tilde\alpha  & = Sq^4(w_{11}w_6w_3)\\
& = w_{11}(w_6w_4w_3 + w_7w_3^2 +  w_6w_5w_2 + w_{10}w_3\\
 & \phantom{aaaa} +w_8w_3w_2+w_6w_4w_3+ w_6w_5w_2 + w_6w_3w_2^2 + w_7w_3^2)\\
 & = w_{11}(w_{10}w_3 + w_8w_3w_2 + w_6w_3w_2^2),\\ 
Sq^3Sq^1 \tilde\alpha  & = Sq^3Sq^1(w_{11}w_6w_3)\\
 & = Sq^3(w_{11}w_7w_3)\\
 & = w_{11}w_7w_3^2,\\
Sq^5\tilde\beta & = Sq^5(w_{11}w_2^4)\\
 & = w_{11}w_5w_2^4,\\
 Sq^8\tilde\gamma & = Sq^8(w_{11}w_3w_2)\\
                            & =  w_{11}(w_8w_3w_2 + w_7w_3^2 + w_6w_5w_2 + w_6w_3w_2^2 + w_6w_3w_2^2\\
  & \phantom{aaaa} + w_5w_3^2w_2 + w_5^2w_3 + w_5w_3^2w_2\\
    & \phantom{aaaa} + w_4w_3^3 + w_5w_4w_2^2+w_4w_3w_2^3+ w_3^3w_2^2)\\
  & = w_{11}(w_8w_3w_2 + w_7w_3^2 + w_6w_5w_2 + w_5^2w_3+ w_5w_4w_2^2+ w_4w_3^3 + w_4w_3w_2^3+ w_3^3w_2^2),\\
Sq^7Sq^1\tilde\gamma & = Sq^7Sq^1(w_{11}w_3w_2)\\
& = Sq^7(w_{11}w_3^2)\\
 & = w_{11}(w_7w_3^2 + w_5^2w_3+w_3^3w_2^2),\\
 Sq^6Sq^2\tilde\gamma & = Sq^6Sq^2(w_{11}w_3w_2) \\
  & = Sq^6(w_{11}(w_3w_2^2+ w_5w_2))\\
 & = w_{11}(w_6w_5w_2 + w_6w_3w_2^2 + w_5^2w_3 + w_5w_4w_2^2 + w_4w_3w_2^3\\
 & \phantom{aaaa} + w_4w_3^3 + w_5w_4w_2^2 + w_5w_4w_2^2 + w_3^3w_2^2 + w_5w_3^2w_2 \\
  & \phantom{aaaa} + w_3^3w_2^2 + w_3w_2^5 + w_9w_2^2 + w_7w_2^3+w_6w_3w_2^2+w_5w_4w_2^2\\
   & \phantom{aaaa} + w_5w_3^2w_2 + w_5w_2^4 + w_5w_2^4 + w_3w_2^5 + w_5^2w_3\\
    & \phantom{aaaa} + w_9w_2^2 + w_7w_2^3 + w_6w_3w_2^2 + w_5w_4w_2^2) \\
    & = w_{11}(w_6w_5w_2 + w_6w_3w_2^2 + w_5w_4w_2^2 + w_4w_3w_2^3 + w_4w_3^3),\\ 
Sq^9\tilde\delta & = Sq^9(w_{11}w_2^2)\\
& = w_{11}(w_9w_2^2 + w_7w_3^2 + w_5w_2^4)
\end{align*}
and obtain the same expression, as the reader can easily check.
\end{proof}

\begin{rem}
The above obstruction was discovered as follows. An integral homology class $z\in H_{n-k}(N)$ is embedded (respectively, immersed) if and only if the Poincar\'e dual class $PD(z)\in H^k(N)$, viewed as a map $N_+\to K(\Z,k)$, admits a lift through the universal Thom class $t_k\colon MSO_k\to K(\Z,k)$ (respectively, the universal Thom class for immersions $\tilde{t}_k\colon QMSO_k\to K(\Z,k)$). In either case the Thom class in question is the bottom Postnikov piece, and so the obstructions to lifting are the pullbacks of the $k$-invariants of the Thom space. Thus the obstructions to embeddability live in cohomology groups $H^{i+1}(N;\pi_i MSO_k)$, while the obstructions to immersability live in $H^{i+1}(N;\pi_i QMSO_k)$. If we want to find something immersed but not embedded, we want an obstruction in the kernel of the homomorphism
\[
 H^{i+1}(N;\pi_i MSO_k) \to H^{i+1}(N;\pi_i QMSO_k) = H^{i+1}(N;\pi_i^S MSO_k)
\]
induced by stabilisation on the coefficients. According to Pastor \cite[Theorems 3.2 and 4.3]{Pastor}, if $k+1\equiv 0(4)$ and neither $k+1$ nor $k+2$ is a power of $2$, then the stabilisation map for $i=2k+1$ fits into a short exact sequence
\[
\xymatrix{
0 \ar[r] & \Z_2 \ar[r] & \pi_{2k+1} MSO_k \ar[r] & \pi_{2k+1} QMSO_k \ar[r] & 0,
}
\]
and $\pi_{2k+1} QMSO_k \cong \Omega_{k+1}$ is the $(k+1)$-st oriented cobordism group. We are therefore interested in such values of $k$, the smallest of which is $k=11$. In this case we have
\[
\xymatrix{
0 \ar[r] & \Z_2 \ar[r] & \pi_{23} MSO_{11} \ar[r] & \Omega_{12}=\Z^3 \ar[r] & 0,
}
\]
showing that $\pi_{23} MSO_{11}\cong \Z_2\oplus \Z^3$. We therefore analysed the Postnikov tower of the space $M:=MSO_{11}$ in the appropriate range, which takes the form
\begin{equation}\label{towerM}
\xymatrix{
 & M[22] \ar[rr]^-{k_{24}^M} \ar[d] && K(\Z_2\oplus\Z^3,24)\\
 & M[21]  \ar[rr]^-{0} \ar[d] && K(\Z_2,23)\\
 & M[20] \ar[rr]^-{0} \ar[d] && K(\Z_2,22) \\
 & M[19] \ar[d] \ar[rr]^-{0} && K(\Z_2\oplus\Z_2,21) \\
 & M[16] \ar[d] \ar[rr]^-{(\beta P^1_5,\frac13\beta P^2_3)} && K(\Z\oplus\Z,20)\\
 & M[15] \ar[d] \ar[rr]^-{0}&& K(\Z_2,17) \\
M \ar[r]^-{t} & K(\Z,11) \ar[rr]^-{\beta P^1_3} && K(\Z,16).
}
\end{equation}
We then found an explicit lift $q[22]\colon M\to M[22]$ giving a $22$-equivalence, and analysed its kernel in mod $2$ cohomology, which (after some lengthy but elementary calculations assisted by Sage) led to the obstruction described above. 
\end{rem}

We now define the manifolds $N$ appearing in Theorem \ref{thm2}. 
Consider the homotopy long exact sequence of the frame bundle of $S^{11}$, $SO_{11} \to SO_{12} \xra{p} S^{11}$:
\[
\xymatrix{
\cdots \ar[r] & \pi_{12}(SO_{11}) \ar[r] & \pi_{12}(SO_{12}) \ar[r]^-{\pi_{12}(p)} & \pi_{12}(S^{11}) \ar[r] & \cdots
}
\]
Kervaire \cite{Kervaire} has calculated $\pi_{12}(SO_{11})\cong \Z_2$ and $\pi_{12}(SO_{12})\cong \Z_2\oplus\Z_2$, using results of Paechter \cite{Paechter}. It follows that there are two elements $\xi \in \pi_{12}(SO_{12})$ such that 
$\pi_{12}(p)(\xi)= \eta$ is a generator of  $\pi_{12}(S^{11}) = \Z_2$.
The elements $\xi$ are clutching functions of rank $12$ vector bundles over $S^{13}$, 
which we also denote by $\xi$, 
and we define the $24$-manifolds $N$ to be the total space of the unit sphere bundles of $\xi$.

The manifolds $N$ were considered by Ishimoto \cite{Ishimoto} as part of his classification of $(n{-}2)$-connected $2n$-manifolds with torsion-free homology up to the action of homotopy $2n$-spheres.
It turns out that both total spaces with $\pi_{12}(p) = \xi$ are diffeomorphic, mod homotopy $24$-spheres,
and this almost diffeomorphism class appears in row ``$4t$ ($t$ odd)" and column ``Type I" of Table 2 on page 213 of \cite{Ishimoto}. 
According to James and Whitehead \cite{JW}, the 
total spaces $N$ have cell structures
\[
N\simeq (S^{11}\cup_\eta e^{13}) \cup e^{24}.
\]
Theorem \ref{thm2} is now restated as follows.

\begin{thm} \label{t:24ibne}
Each generator $ z \in H_{13}(N)$ is immersed but not embedded.
\end{thm}

\begin{proof}
We first show that $z$ is not embedded, using Theorem \ref{thm3}. Let $x=PD(z)\in H^{11}(N)$. Since the inclusion of the $13$-skeleton of $N$ induces isomorphisms of cohomology in degrees less than $24$, and since $Sq^2$ detects $\eta$, one sees that $Sq^2(x)\in H^{13}(N;\Z_2)\cong\Z_2$ is the generator. By Poincar\'e duality, $xSq^2 x\neq 0$. The Adem relation $Sq^1Sq^{10}=Sq^{11}$ and the fact that $H^{23}(N;\Z_2)=0$ then gives
\[
xSq^2 x + Sq^{11}Sq^2 x = x Sq^2 x \neq 0.
\]
However, the cohomology groups $H^{20}(N;\Z_2)$, $H^{19}(N)$, $H^{16}(N;\Z_2)$ and $H^{15}(N)$ are all zero, hence the left-hand side of the expression in Theorem \ref{thm3} is always zero, contradicting the existence of the classes $\alpha, \ldots , \delta$.

Finally we show that $z$ is immersed, by a similar argument as in the $\symp$ case. First, note that $N$ is stably parallelizable: Since $\pi_{12}(SO)=0$, each bundle $\xi$ is stably trivial. Consider the sequence of embeddings 
\[
\begin{tikzcd}        
    N=S(\xi) \arrow[hook]{r} & S(\xi\oplus\R)   \arrow[hook]{r} &  S(\xi\oplus\R^2) \arrow[hook]{r} &\cdots  \arrow[hook]{r} & S(\xi\oplus\R^\ell), 
\end{tikzcd}
\]
all with trivial normal bundle. For $\ell\gg0$ the manifold $S(\xi\oplus \R^\ell)\cong S^{13}\times S^{11+\ell}$ is stably parallelizable, hence so is $N$.

 It follows that the top cell of $N$ splits off stably to give
\[
\Sigma^\ell N \simeq \Sigma^\ell K \vee S^{24+\ell}
\]
for $\ell\gg0$, where $K=S^{11}\cup_\eta e^{13}$ denotes a $13$-skeleton of $N$.  As in the previous case, since $\eta$ is in the image of the $J$-homomorphism $J\colon\pi_1(SO_{11})\to \pi_{12}(S^{11})$ we find that $K$ is the Thom complex of the non-trivial rank $11$ vector bundle $\zeta$ over $S^2$, with Thom class $t(\zeta)=x|_K$. As before we conclude the existence of a stable map $f\colon N_+\rightsquigarrow MSO_{11}$ such that $f^*t=x$.
\end{proof}


We next give an alternative proof of the immersability of $z\in H_{13}(N)$ by analyzing the Postnikov tower of $QMSO_{11}$ using results of Pastor \cite{Pastor} and Trou\'e \cite{Troue}.
Let $\iota\colon MSO_k\to QMSO_k$ denote the stabilisation map. The long exact sequence \cite[(2.4)]{Pastor}, which is valid for 
$n<2k{-}1$, may be written as
\[
\xymatrix{
\cdots \ar[r] & \Omega_{n-k+1}^{w(k)} \ar[r] & \pi_{n+k}MSO_k \ar[r]^{\iota_*} & \pi_{n+k}QMSO_k \ar[r]^-{D} & \Omega^{w(k)}_{n-k} \ar[r] & \cdots~,
}
\] 
where $\Omega_*^{w(k)}$ are certain bordism groups whose values are known for $*\leq 2$, and $D$ is represented by taking double points. (Recall that the groups $\pi_{n+k}MSO_k$ and $\pi_{n+k}QMSO_k$ correspond to bordism class of embeddings $M^{n}\hookrightarrow \R^{n+k}$ and immersions $M^n\imm \R^{n+k}$, repsectively.) It is immediate that $\iota_*$ is an isomorphism for $n<k{-}1$ and an epimorphism for $n=k{-}1$. It is also an isomorphism for $n=k{-}1$, since the previous map in the sequence
\[
D\colon \pi_{2k}QMSO_k \to \Omega_0^{w(k)}
\]
admits a splitting $\partial\colon \Omega_0^{w(k)}\to \pi_{2k}QMSO_k$ given by an immersion $j_k\colon S^k\imm \R^{2k}$ with a single double point (see \cite[Theorem 3.1]
{Pastor}). For $n=k$ and specialising to $k\equiv 3(4)$, since $\Omega_1^{w(k)}=0$ and $\pi_{2k}MSO_k\cong \Omega_k$ \cite[Theorem 4.2]{Pastor} there is a split short exact sequence
\[
\xymatrix{
0 \ar[r] & \pi_{2k}MSO_k \ar[r]^{\iota_*} \ar@{=}[d] & \pi_{2k}QMSO_k \ar[r]^-{D} \ar@{=}[d] & \Z_2 \ar@{=}[d] \ar[r] & 0\\
0 \ar[r] & \Omega_{k}\ar[r] & \Omega_{k}\oplus\Z_2 \ar[r]  & \Z_2 \ar[r] & 0.
}
\]
Summarising, for all $k$ the map $\iota\colon MSO_k\to QMSO_k$ induces an isomorphism on homotopy groups in degrees less than $2k$, and if in addition $k\equiv 3(4)$ it induces a split monomorphism in degree $2k$. 

We now specialise to $k=11$, and apply naturality of Postnikov towers as decribed for example in \cite[Theorem 2.1]{Kahn}. The space $Q:=QMSO_{11}$ admits a Postnikov tower such that $\iota\colon M\to Q$ induces equivalences of Postnikov sections with the tower for $M$ displayed as (\ref{towerM}) in the last section, up to $\iota[21]\colon M[21]\to Q[21]$. Further, the $k$-invariant $k^Q_{23}\in H^{23}(Q[21];\pi_{22}(Q))$ satisfies
\[
\iota[21]^*k^Q_{23} = \pi_{22}(\iota)_* k^M_{23}=\pi_{22}(\iota)_* 0=0 \in H^{23}(M[21];\pi_{22}(M)),
\]
hence $k^Q_{23}=0$. Hence there is a homotopy equivalence 
$Q[22]\simeq M[22]\times K(\Z_2,22)$, with the map $\iota[22]\colon M[22]\to Q[22]$ given by inclusion of a factor.  The next homotopy group is $\pi_{23}(Q)\cong \Omega_{12}\cong \Z^3$ by \cite[Theorem 3.2]{Pastor}, and the induced map $\pi_{23}(\iota)\colon\pi_{23}(M)\to \pi_{23}(Q)$ is the split surjection $\Z_2\oplus\Z^3\to \Z^3$. The situation is summarized in the following diagram:
\[
\xymatrix{
K(\Z^3,24) &&Q[22] \ar[ll]_-{k_{24}^Q} \ar[d] &   M[22] \ar[rr]^-{k_{24}^M} \ar[l]_{\iota[22]}\ar[d] && K(\Z_2\oplus\Z^3,24)\\
K(\Z_2\oplus\Z_2,23) &&Q[21] \ar[ll]_-{k_{23}^Q=0} \ar[d]  & M[21]  \ar[rr]^-{k_{23}^M=0} \ar[l]_{\iota[21]}^{\simeq}\ar[d] && K(\Z_2,23)\\
&& \vdots  & \vdots &&  
}
\]

We now wish to understand the next $k$-invariant 
\[
k^Q_{24}\in H^{24}(Q[22];\pi_{23}(Q)). 
\]
By the K\"unneth formula and the fact that $H^{24}(K(\Z_2,22))=0$, the inclusion $\iota[22]\colon M[22]\to Q[22]\simeq M[22]\times K(\Z_2,22)$ induces an isomorphism 
\[
\xymatrix{
\iota[22]^*: H^{24}(Q[22];\Z^3) \ar[r]^-{\cong} & H^{24}(M[22];\Z^3).
}
\]
 Note that again by \cite[Theorem 2.1]{Kahn}, we have
\[
\iota[22]^*(k^Q_{24})=\pi_{23}(\iota)_*(k^M_{24})\in H^{24}(M[22];\Z^3)
\]
and so the orders of the components of $k^Q_{24}$ equal the orders of the components in integral cohomology of $k^M_{24}$. Now for $\ell\gg0$ set $\Omega:=\Omega^\ell MSO(11+\ell)$. Note that 
\[
\pi_i(\Omega)= \pi_{i+\ell}(MSO(11+\ell)) = \pi_{i-11}(MSO)
\]
for $\ell$ sufficiently large. Set $f\colon M\to \Omega$ to be the adjoint of 
the natural map $\Sigma^\ell MSO_{11}\to MSO(11{+}\ell)$; by \cite{Pastor} (for example) the map $\pi_i(f)\colon \pi_i(M)\to \pi_i(\Omega)$ is an isomorphism for $i<23$ and an epimorphism for $i=23$. In the induced map of Postnikov towers the map $f[22]\colon M[22]\to \Omega[22]$ is a homotopy equivalence, and $\pi_{23}(f)_*(k^M_{24})=f[22]^*(k^\Omega_{24})$, and so the orders of the integral components of $k^M_{24}$ equal the orders of the corresponding components of $k^\Omega_{24}$. However, the $k$-invariants of $\Omega$ are obtained from the $k$-invariants of $MSO(11+\ell)$ by applying the cohomology suspension homomorphism $\ell$ times. Therefore their orders divide the orders of the corresponding $k$-invariants for the spectrum $MSO$, which were calculated by Trou\'e \cite{Troue}; in  dimension $12$ the $k$-invariants for $MSO$ have orders $27$, $15$ and $7$. 

By the above analysis and standard obstruction-theoretic arguments, we have the following generalisation of the immersability statement of Theorem \ref{t:24ibne}.

\begin{prop}\label{kQ}
Let $N$ be a closed oriented manifold of dimension $24$, and let $z\in H_{13}(N)$ be a homology class with Poincar\'e dual $x\in H^{11}(N)$. If $\beta P_3^1(x)$, $\beta P_3^2(x)$ and $\beta P_5^1(x)$ all vanish, then $z$ is immersed.
\end{prop}

It is known that the $k$-invariants of the spectrum $MSO$ are all (odd) torsion. Outside of the stable range, however, very little is known about the Postnikov towers of the spaces $MSO_k$ and $QMSO_k$. In particular, we do not know if the $k$-invariants are all torsion, which prevents us from deducing that every homology class in a manifold with torsion-free homology must be embedded or immersed.

\section{Proof of Theorem D}\label{S:SRnotim}
In this section we present examples of integral homology classes in manifolds which are Steenrod representable but not immersed. That such examples should exist is not surprising, after the work of 
the second author and Sz\H{u}cs \cite{GSz}, which shows that there are mod $2$ homology classes which are not immersed, despite every mod $2$ class being Steenrod representable. We include such examples here to illustrate that an extra argument is needed to conclude that  the generators of $H_7(\symp)$ are immersed.

To begin, observe that if $z\in H_*(M)$ is immersed then so is its mod $2$ reduction $\rho_2(z)\in H_*(M;\Z_2)$, by forgetting orientation. Thus our strategy is to look for an integral class whose mod $2$ reduction is not immersed, for which we can apply the obstructions of \cite{GSz}.

We recall some notation around the mod $2$ Steenrod algebra. For a sequence of non-negative integers $I=(i_1,\ldots , i_r)$ we write $Sq^I=Sq^{i_1}\cdots Sq^{i_r}$. We say that $I$ is \emph{admissible} if $i_j\ge 2i_{j+1}$ for $j=1,\ldots , r$ (where we understand $i_{r+1}=0$). The \emph{excess} of $I$ is defined by $e(I)=\sum_{j=1}^r (i_j-2i_{j+1})$. Observe that if $x\in H^k(X;\Z_2)$ and $I$ is of excess $k$, then $i_1=k+\sum_{j=2}^r i_j$ and therefore
\[
Sq^I(x) = Sq^{i_1}\big( Sq^{i_2}\cdots Sq^{i_r}(x)\big) = \big( Sq^{i_2}\cdots Sq^{i_r}(x)\big)^2.
\]
A classical theorem of J.-P.\ Serre \cite{Serre}
describes the mod $2$ cohomology algebra of $\Z_2$-Eilenberg--MacLane spaces as
\[
H^*(K(\Z_2,k);\Z_2) \cong \Z_2 [Sq^I(\iota_k) \mid I\mbox{ admissible of excess less than }k],
\]
where $\iota_k\in H^k(K(\Z_2,k);\Z_2)$ is the fundamental class.

\begin{thm}[{\cite[Theorem 1.2]{GSz}}]
Let $k>1$, and let $I$ have excess $e(I)=k$. If $z\in H_{n-k}(N;\Z_2)$ is an immersed homology class with Poincar\'e dual $x=PD(z)\in H^k(N;\Z_2)$, then $Sq^I(x)$ is the reduction of an integral class. In other words, if $\beta_2 Sq^I(x)\neq 0$ where $\beta_2$ denotes the 
integral Bockstein operator, then $z$ is not immersed.
\end{thm}

For this to give useful obstructions to immersability, it must be shown that $\beta_2 Sq^I(\iota_k)\neq 0$ for some sequences $I$ of excess $k$, where $\iota_k\in H^k(K(\Z_2,k);\Z_2)$ is the fundamental class. This is achieved in \cite[Section 3]{GSz} by referencing Browder's results \cite{Browder} on the mod $p$ Bockstein spectral sequence for $K(\Z_{p^s},k)$. In particular, we have the following.

\begin{prop}[{\cite[Theorem 3.2]{GSz}, corollary to \cite[Theorem 5.5]{Browder}}]
Let $k>1$, and let $J$ be an admissible sequence of excess less than $k$ such that $j_1\neq 1$ and $|J|+k$ is even. Then $(Sq^J\iota_k)^2\in H^{2|J|+2k}(K(\Z_2,k);\Z_2)$ is not the reduction of an integral class.
\end{prop}

For example, when $k=3$ we may take $J=(2,1)$, of excess $1$. The result in this case tells us that $\beta_2 Sq^6Sq^2Sq^1\iota_3\neq 0$. If $K(\Z_2,3)$ were a manifold, this would tell us that the homology class Poincar\'e dual to $\iota_3$ is not immersed. Note that it would also tell us that the class dual to $Sq^1(\iota_3)$ is not immersed, since $(6,2)$ is admissible of excess $4$. 

We are now in a position to prove Theorem~\ref{thm4} from the Introduction.

\begin{thm}
For all $n \geq 27$, there exist closed oriented $n$-manifolds $N$ and $2$-torsion classes $z \in H_{n-4}(N)$,
which are Steenrod representable but not immersed.
\end{thm}

\begin{proof}
We produce such examples using the standard technique of thickenings. We embed the $14$-skeleton $K=K(\Z_2,3)^{(14)}$ in a high-dimensional Euclidean space $\R^{n+1}$ with a regular neighbourhood $W$. This $W$ is a compact $(n{+}1)$-manifold homotopy equivalent to $K$, whose boundary is a closed smooth $n$-manifold $N:=\partial W$. By Poincar\'e--Lefschetz duality $$H^{13}(W,N)\cong H_{n-12}(W)\cong H_{n-12}(K),$$ which vanishes provided $n\ge 27$. It follows that for $n\geq 27$ the induced map on cohomology $$H^{13}(K)\cong H^{13}(W)\to H^{13}(N)$$ is injective. Thus we get a closed orientable manifold $N$ with a class $X\in H^3(N;\Z_2)$ such that $\beta_2 Sq^6Sq^2Sq^1(X) \neq 0$.  


Now set $x:=\beta_2 (X) \in H^4(N)$. The Poincar\'e dual class $z\in H_{n-4}(N)$ is not immersed, since its mod $2$ reduction is the Poincar\'e dual of $Sq^1(X)$, whose immersability is obstructed by the operation $\beta_2 Sq^6Sq^2$. However, since $z$ is $2$-torsion it is Steenrod representable, see for example \cite[(15.4)]{C-F}.
\end{proof}

The same method will produce Steenrod representable but not immersed homology classes in any codimension greater than $4$, but does not seem to produce codimension $3$ examples, as for $k=2$ there do not exist sequences $J$ of excess less than $2$ with $|J|+2$ even.

\section{Representing immersions and their double points}\label{S:Whitney}
In this section we use Whitney's self-intersection formula \cite{Whitney, Herbert} to show that for each of the immersed but not embedded homology classes in Theorems \ref{thm1} and \ref{thm2}, the homology class of the double points in the source manifold of a representing immersion gives an obstruction to embeddability. In the case of Theorem \ref{thm1}, this gives an alternative proof of the Bohr--Hanke--Kotschick result \cite{BHK}, that each generator $z \in H_7(\symp)$ is not embedded.
We also describe formal immersions representing these homology classes.

\subsection{Whitney's self-intersection formula}
Let $f\colon M^{n-k}\imm N$ be a self-transverse immersion of closed manifolds. The multiple-point set $\{x\in M \mid |f^{-1}f(x)|>1\}\subseteq M$ is then the image of an immersion $\mu_2(f)\colon\Delta_2(f)\imm M$, where
\[
\Delta_2(f):=\{(x,y)\in M\times M \mid x\neq y, f(x)=f(y)\},
\qquad \text{and} \qquad
\mu_2(f)(x,y) := x.
\]
To see this, let $F(M,2) := \{(x,y)\in M\times M \mid x\neq y\}$ be the ordered configuration space, and form the diagram below in which the square is a transverse pullback:
\[
\xymatrix{
\Delta_2(f) \ar[r]^-{i} \ar[d]_{\psi_2(f)} & F(M,2) \ar[r]^-{\mathrm{pr}_1} \ar[d]^{(f\times f) |_{F(M,2)}} & M \\
N \ar[r]^-{\Delta_N} & N\times N & 
}
\]
Although $F(M,2)$ is an open $(2n{-}2k)$-manifold, compactness guarantees that $\Delta_2(f)$ is a closed $(n{-}2k)$-submanifold. The composition along the top is $\mu_2(f)$; it is an immersion since $f\circ \mu_2(f)=\psi_2(f)$ is an immersion (as the transverse pullback of an immersion). The normal bundle of $\mu_2(f)$ satisfies 
\[
\mu_2(f)^*\nu_f\oplus \nu_{\mu_2(f)}\cong \nu_{\psi_2(f)}\cong i^*\nu_{(f\times f) |_{F(M,2)}}.
\]
It follows that $\nu_{\mu_2(f)}$ is orientable if $\nu_f$ is orientable.

Under our assumption that $M$ and $N$ are orientable, the above considerations show that $\mu_2(f)\colon\Delta_2(f)\imm M$ represents a homology class $\mu_2(f)_*[\Delta_2(f)]\in H_{n-2k}(M)$, whose Poincar\'e dual cohomology class we denote by $m_2(f)\in H^{k}(M)$.

\begin{prop}[Whitney {\cite{Whitney, Herbert}}] \label{p:Wh}
Let $x\in H^k(N)$ be the Poincar\'e dual of the homology class represented by $f\colon M^{n-k}\imm N^n$. Then 
\[
f^*(x) = e(\nu_f) + m_2(f)\in H^k(M),
\]
where $e(\nu_f)$ denotes the Euler class of the normal bundle of $f$.
\qed
\end{prop}

\begin{thm}\label{t:doublesymp}
Let $f:M^7\imm\symp$ be any self-transverse immersion representing a generator $z\in H_7(\symp)$. 
Then $0\neq m_2(f)\in H^3(M)$. In particular, $z$ is not embedded.
\end{thm}

\begin{proof}
Let $x\in H^3(\symp)$ be the Poincar\'e dual of $z=f_*[M]$, and $x_3:=\rho_3(x)\in H^3(\symp;\Z_3)$ its mod $3$ reduction. As we have seen, $x_3\cup P^1_3(x_3)\in H^{10}(\symp;\Z_3)$ is nonzero. Therefore
\begin{align*}
0 \neq \langle x_3\cup P^1_3(x_3),[\symp]\rangle & = \langle P^1_3(x_3), x\cap[\symp]\rangle \\
& = \langle P^1_3(x_3),f_*[M]\rangle\\
& = \langle f^* P^1_3(x_3),[M]\rangle,
\end{align*}
which implies that $f^* P^1_3(x_3)=P^1_3(f^*(x_3))\neq 0$, and therefore $0\neq f^*(x_3)\in H^3(M;\Z_3)$. It follows that $2f^*(x_3)\neq 0$. 

Now by Whitney's formula
\[
2f^*(x) = 2e(\nu_f) + 2m_2(f),
\]
and since the rank of $\nu_f$ is odd, $e(\nu_f)=-e(\nu_f)$ and so $2e(\nu_f)=0$. Reducing the above mod $3$ we therefore get that 
\[
0\neq 2f^*(x_3)=\rho_3(2m_2(f)),
\]
which entails that $m_2(f)\neq 0$.
\end{proof} 


\begin{thm}\label{t:doubleN}
Let $N$ be a $24$-manifold as in Theorem~\ref{thm2}
and let $f:M^{13} \imm N$ be any self-transverse immersion representing a generator $z \in H_{13}(N)$. 
Then $0 \neq m_2(f)\in H^{11}(M)$. In particular, $z$ is not embedded. \qed
\end{thm}

\begin{proof}
With $x_2\in H^{11}(N;\Z_2)$ the mod $2$ reduction of the Poincar\'e dual of $z$, we have 
\[
0\neq  \langle x_2 Sq^2(x_2),[N] \rangle = \langle Sq^2(x_2),f_*[M]\rangle = \langle Sq^2(f^*(x_2)),[M]\rangle,
\]
and so $Sq^2(f^*(x_2))\neq 0$. Applying $Sq^2\circ\rho_2$ to Whitney's formula gives
\[
0\neq Sq^2(f^*(x_2)) = Sq^2(w_{11}(\nu_f)) + Sq^2\rho_2(m_2(f)).
\]
Using that $N$ is stably parallelizable, 
we deduce that $\nu_f$ represents the stable normal bundle $\nu_M$ of $M$. 
Then since $M$ is orientable, Wu's formula gives
\[
Sq^2(w_{11}(\nu_f))=Sq^2(w_{11}(\nu_M))=w_2(M)w_{11}(\nu_M).
\]
%
By a theorem of Massey and Peterson, \cite[Theorem I(i)]{MP}, $w_{11}(\nu_M) = 0$ for any closed $13$-manifold $M$. Thus $Sq^2\rho_2(m_2(f))\neq 0$, and $m_2(f)\neq 0$ as claimed.
\end{proof}

\begin{rem}
For the examples in Theorem~\ref{thm2}, any representing immersion $f \colon M \imm N$ has non-trivial Hatcher-Quinn invariant \cite[Theorem 2.3]{HQ}, since it is not regularly homotopic to an embedding.
It is natural to ask about the relationship between the homology class $PD_M m_2(f) \in H_{n-2k}(M)$ and
the Hatcher-Quinn invariant of $f$, but we leave that question for future work.
\end{rem}

\subsection{Some explicit formal immersions}
In this subsection we present explicit formal immersions representing the generators of $H_7(\symp)$ and $H_{13}(N)$.
By Hirsch-Smale theory \cite{Hirsch}, these formal immersions are homotopic to immersions.

Recall from the proof of Proposition \ref{Sp2} that the $7$-skeleton $K$ of $\symp$ is the Thom space of a rank $3$ vector bundle $\xi\to S^4$. Let $M=S(\xi\oplus\eps)$ be the sphere bundle of the stabilisation of $\xi$; it is the total space of an $S^3$-bundle over $S^4$. Let $s\colon S^4\to M$ be one of the obvious sections. Then the quotient $M/s(S^4)$ is identified up to homotopy with the Thom space of $\xi$, and under this identification the zero section of $\xi$ corresponds to the image of the other obvious section $-s \colon S^4\to M$. We now define $f'\colon M\to \symp$ to be the composition
\[
M \to M/s(S^4) \simeq K \subseteq \symp.
\]
By construction, $f'_*[M]$ generates $H_7(\symp)\cong H_7(K)$. 

We now use Hirsch-Smale theory 
\cite{Hirsch} to find an immersion $f\colon M \imm \symp$ homotopic to $f'$. We only need show that 
$f'\colon M \to \symp$ is covered by a formal immersion, in other words that there is a fibrewise linear inclusion
$TM \to  (f')^*T\symp$. Since $\symp$ is a Lie group and therefore parallelizable, this is equivalent to
finding an embedding $TM \to  \eps^{10}$, which by Hirsch-Smale again is equivalent to $M$ immersing into $\R^{10}$. Now we appeal to \cite[Theorem 4.1]{Wilkens}, which asserts that any $2$-connected $7$-manifold immerses in $\R^{10}$. 

We can give an alternative proof that $m_2(f)\neq 0$ for this immersion. Observe that $M$ being $2$-connected, the Hurewicz homomorphism $\pi_3(M)\to H_3(M)$ is onto, and any homology class $y\in H_3(M)$ may be written as $y=\alpha_*[S^3]$. Then
\[
\langle e(\nu_f), y\rangle = \la e(\nu_f),\alpha_*[S^3] \ra = \la\alpha^* e(\nu_f),[S^3] \ra = \la e(\alpha^*\nu_f),[S^3]\ra = 0,
\]
where the last equality follows since $\pi_2(SO(3))=0$ implies that any oriented rank $3$ vector bundle over $S^3$ is trivial. Hence $e(\nu_f)=0$ (clearly, the same is true of any immersion $f\colon M\imm\symp$). We also note that $f^*(x)\neq 0$, where $x\in H^3(\symp)\cong H_7(\symp)$ is a generator. This is because $x|_K\in H^3(K)$ is a Thom class, which is represented by the zero section $S^4\to (D\xi,S\xi)$, and therefore by construction $f^*(x)$ is represented by the section $-s\colon S^4\to M$. By the injectivity of the induced map $(-s)_* \colon H_4(S^4)\to H_4(M)$, this class is nonzero. Applying Whitney's formula now gives $m_2(f)\neq 0$.

In conclusion, the immersion $f\colon M\imm \symp$ represents the generator $z\in H_7(\symp)$, and the multiple-points of $f$ are represented by the image of $S^4$ in $M$ under a section, thus are non-homologous to zero.

Similar arguments apply to represent the class $z\in H_{13}(N)$ appearing in Theorem \ref{thm2}. We define $M := S(\zeta\oplus\eps)$, where $\zeta$ is the rank $11$ bundle over $S^2$ appearing in the proof of Theorem \ref{t:24ibne}. Since $N$ is stably parallelizable, it is almost parallelizable, and the analogous map $f' \colon M \to N$ misses a point. Therefore $f'$ can be covered by a formal immersion if $M$ immerses in $\R^{24}$. Here $\zeta \oplus \eps \to S^2$ is a rank $12$ vector bundle; let $\xi\to S^2$ be a stable inverse to $\zeta\oplus\eps$, then since $\pi_1(SO(2))\to \pi_1(SO)$ is onto we may choose $\xi$ to be a rank $2$ bundle.  It follows that $M = S(\zeta\oplus\eps)$ embeds in $S^2\times \R^{14}$, hence in $\R^{17}$.

As above, we can show that $f^*(x)\neq 0$ where $x\in H^{11}(N)$ is a generator, and that $e(\nu_f)=0$ since the Euler class of any rank $11$ bundle over $S^{11}$ is trivial. It follows that $m_2(f)\neq 0$. Here the double points are represented by a section $S^2\to V$.

\end{document}